\documentclass[11pt]{article}
\usepackage{graphicx}
\usepackage{amscd}
\usepackage{amsmath}
\usepackage{caption}
\usepackage{amsfonts}
\usepackage{amssymb}
\usepackage{amsthm}
\usepackage{mathrsfs}
\usepackage{multicol}
\usepackage{color}
\usepackage[english]{babel}
\usepackage[T1]{fontenc}
\usepackage{textcomp}
\usepackage[utf8]{inputenc}
\usepackage{enumitem}
\usepackage{indentfirst}
\usepackage[title,toc]{appendix}
\usepackage[inner=2cm, outer=1.5cm, top=3cm, bottom=3cm]{geometry}
\usepackage{tikz}
\usetikzlibrary{positioning,fit,arrows.meta,shapes,matrix}
\usepackage{float}

\newcommand{\N}{\mathbb N}

\newcommand{\RR}{{{\rm I} \kern -.15em {\rm R} }}

\author{Chiara Cicolani\footnote{Fakultät für Mathematik und Naturwissenschaften, University of Wuppertal, Gaußstraße 20, 42119, Wuppertal, Germany. Email: cicolani@uni-wuppertal.de.}, Elisa Continelli\footnote{Dipartimento di Matematica "Tullio Levi-Civita", Università degli Studi di Padova, Via Trieste 63, 35121 Padova, Italy. Email: elisa.continelli@unipd.it.}, and Cristina Pignotti\footnote{Dipartimento di Ingegneria e Scienze dell'Informazione e Matematica, Università degli Studi dell'Aquila, Via Vetoio, Loc. Coppito, 67100 L'Aquila, Italy. Email: cristina.pignotti@univaq.it.}}
\title{Asymptotic consensus and flocking under decaying persistent excitation on rooted digraphs}
\date{\today}

\begin{document}

	\theoremstyle{plain} \newtheorem{thm}{Theorem}[section] \newtheorem{cor}[thm]{Corollary} \newtheorem{lem}[thm]{Lemma} \newtheorem{prop}[thm]{Proposition} \theoremstyle{definition} \newtheorem{defn}{Definition}[section] 
	
	\newtheorem{oss}[thm]{Remark}
	\newtheorem{ex}{Example}[section]
	\newtheorem{lemma}{Lemma}[section]

\maketitle
\begin{abstract}
In this paper, we investigate first- and second-order alignment models with non-universal interaction, time delays and possible communication failures, extending the results in \cite{CCP} to interaction digraphs that are only assumed to be rooted and to a weaker Persistence Excitation Condition. In particular, we allow the amount of interaction over time intervals of fixed length to decay polynomially in time.

For the first-order Hegselmann--Krause type model, we prove asymptotic convergence to consensus under a suitable condition relating the decay exponent of the communication weights to the maximal distance from the root. For the second-order Cucker--Smale model, we establish asymptotic flocking under an additional assumption on the decay of the influence function. These results show that collective behavior can still emerge under progressively weakening communication and without requiring strong connectivity of the interaction graph.
\end{abstract}
\section{Introduction}
\setcounter{equation}{0}

Multiagent systems have been deeply investigated in recent years, due to their
wide application to several scientific disciplines, among others biology
\cite{Cama}, economics \cite{Marsan}, robotics \cite{Bullo}, control
theory \cite{Borzi,PRT,WCB}, and social sciences
\cite{Bellomo, CF, Campi}. Among them, there is the
Hegselmann--Krause opinion formation model \cite{HK} and its second-order
version, the Cucker--Smale model, introduced in \cite{CS1} for the description
of flocking phenomena, such as the flocking of birds, the swarming of bacteria
and the schooling of fish (see also \cite{Ha1,Ha2}).  Typically, for the solutions of the
aforementioned multiagent systems, the convergence to consensus, in the case of
the Hegselmann--Krause model, and the exhibition of asymptotic flocking, in the
case of the Cucker--Smale model,  are investigated.

In multiagent systems, it is important to consider the presence of time delay
effects. Indeed, in the applications, one has to take into account certain time
lags due to the propagation of information among the agents or to reaction
times. The analysis of the Hegselmann--Krause model and the Cucker--Smale model
in the presence of time delays, which can be constant or, more realistically,
varying in time, has been carried out by many authors
\cite{CH,CL,CPP, DH, DHY,
	H, H3, HM, HaskCartabia, LW, Lu,P, PR,PT}.
Most of these results require an upper bound on the time delay size in order
to obtain convergence to consensus or flocking. However, asymptotic alignment
results without requiring any smallness condition on the time delay have been
obtained in \cite{H,Cartabia}. Generalizing and extending these
arguments, exponential convergence to consensus in the presence of
time-variable delays has been proved in
\cite{ContPi,Cont} for the first- and second-order
models, respectively, without assuming the time delay size to be small. 

It may happen that the agents involved in an opinion formation or flocking
process are not able to exchange information with all the other agents of the
system, namely each agent can influence or can be influenced only by some
agents. In this case, we are in the presence of a non-universal interaction.
To deal with this kind of interaction, a network topology over the structure
of the model has to be considered.

Another possible scenario is the one in which the system's agents sometimes
suspend the interactions they have with the agents they are linked to. As a
consequence, there is a temporary lack of connection between the system's
elements that may hinder convergence to consensus, for the first-order model,
or flocking, for the second-order one. Therefore, it is important to identify
suitable conditions on the communication weights guaranteeing the asymptotic
alignment of the system.

In \cite{Bonnet}, convergence to consensus and asymptotic flocking for a
class of Cucker--Smale systems under communication failures, namely with
interaction weights possibly degenerating among the agents, have been proved
under suitable assumptions in the case of symmetric interaction coefficients.
Convergence to consensus for a first-order alignment system involving
pair-dependent weights that can eventually degenerate has also been proved in
\cite{AnconaRossi} under a Persistence Excitation Condition. In the case of
nonsymmetric interaction coefficients, exponential convergence to consensus
for the Hegselmann--Krause model with time delay and possible communication
failures has been obtained in \cite{ContPi}. 

More recently, in \cite{CCP}, we considered first- and
second-order alignment models with pair- and time-dependent time delays,
non-universal interaction and possible communication failures. Under the
Persistence Excitation Condition
\begin{equation}\label{classicalPE}
	\int_t^{t+T}\alpha_{ij}(s)\,ds\geq {\tilde\alpha},
	\qquad t\geq0,
\end{equation}
for some $T,\tilde \alpha >0,$
for all the active edges of the interaction graph, exponential convergence to
consensus and exponential flocking were established whenever the digraph describing the interaction among the agents is strongly connected. The
results were obtained for general influence functions, without symmetry or
monotonicity assumptions, and without requiring any smallness conditions on the
time delay size.

{Multiagent systems with time delay and communication failures have been also analyzed in the recent works \cite{Cont2, JWL}.} For further recent results for multi-agent models with time-delay we quote \cite{ ChoiCicoPi, ChoiCP, Cico, COP, CicoPi, CHP, Posneg, Dong}. See also \cite{BCGR, BPDS} for other results in the nondelayed setting. 

The aim of the present paper is to weaken both the Persistence Excitation
Condition and the assumption on the network topology. More precisely, we
replace \eqref{classicalPE} with a generalized Persistence Excitation
Condition, allowing the amount of communication among connected
agents over intervals of fixed length to progressively decrease with time (see \eqref{PE}).
Thus, differently from the classical Persistence Excitation Condition, the
interaction is not required to remain uniformly persistent during the whole
evolution. At the same time, the rootedness assumption considerably weakens
the connectivity requirement: it is sufficient that there exists one agent,
the root, from which all the other agents can be reached through a directed
path.

For the first-order Hegselmann--Krause type model, we prove asymptotic
convergence to consensus under these weaker assumptions.
Therefore, consensus is still achieved even though the communication becomes
progressively weaker and the interaction graph is only rooted. Differently
from the uniform Persistence Excitation case considered in
\cite{CCP}, where exponential convergence was obtained,
the weaker persistent excitation condition leads to asymptotic convergence to
consensus.

We then extend the analysis to the second-order Cucker--Smale model. In this
case, besides velocity alignment, the uniform boundedness of the relative
positions has to be established. This requires a suitable relation between the
decay of the communication weights and the behavior of the influence function.
	This condition shows explicitly the interplay between the deterioration of the
	communication in time, the decay of the influence
	function,  and the topology of the communication network,
	through the maximal distance $\gamma_r$ from the root.
	
	The results obtained here show, in particular, that neither a uniform
	Persistence Excitation Condition nor strong connectivity of the interaction
	graph is necessary to guarantee the emergence of collective behavior. Both
	asymptotic consensus and asymptotic flocking may still occur when the
	communication among connected agents progressively deteriorates and the
	underlying digraph is only rooted, provided that the deterioration is
	sufficiently slow with respect to the network topology and, in the
	second-order case, to the decay of the influence function. As in
	\cite{CCP}, no smallness assumption on the time delay
	size is required.
	
	The rest of the paper is organized as follows. In Section~\ref{2}, we give some preliminary definitions and remarks. In Section~\ref{timedelay}, we analyze the
	first-order Hegselmann--Krause model and establish asymptotic convergence to
	consensus under the generalized Persistence Excitation Condition. In
	Section~\ref{4}, we consider the second-order Cucker--Smale model and
	prove asymptotic flocking under the additional assumption on the influence function.

\section{Preliminary notions}\label{2}
\setcounter{equation}{0}
  In this section, we collect some definitions and notions that will be used in the paper. We start by recalling some notions from graph theory.
\begin{defn}
	Let $ \mathcal{G}=(\mathcal{V}, \mathcal{E})$ be a direct graph consisting of a finite set $\mathcal{V}$ of vertices and a set $\mathcal{E} \subset \mathcal{V} \times \mathcal{V}$ of arcs. The adjacency matrix of $\mathcal{G}$ is the $N\times N$ matrix $(\chi_{ij})_{i,j\in\mathcal{V}}$,
	$$\chi_{ij}=\begin{cases}
		1,\quad (j,i)\in \mathcal{E},\\
		0,\quad (j,i)\notin \mathcal{E},
	\end{cases}\quad \forall i,j\in \mathcal{V}.$$
	For each vertex $i\in \mathcal{V}$, the in-neighborhood of $i$ is \begin{equation}\label{N_i}
		\mathcal{N}_i:=\{j\in \mathcal{V}:(j,i)\in \mathcal{E}\}=\{j= 1,\dots,N:\chi_{ij}=1\},
	\end{equation} 
	and the in-degree of $i$ is \begin{equation}\label{cardN_i}
		N_i:=|\mathcal{N}_i|.
	\end{equation}
	A path in $\mathcal{G}$ from $i_0\in \mathcal{V}$ to $i_p\in\mathcal{V}$ is a finite sequence $i_0, i_1, \dots, i_p$ of distinct vertices $i_k\in \mathcal{V}$, $k=0,\dots,p$ such that $(i_k,i_{k+1})\in \mathcal{E}$, for all $k=0,\dots,p-1$. The integer p is called length of the path. 
	
	A vertex $j\in \mathcal{V}$ is said to be reachable from a vertex $i\in \mathcal{V}$ if there exists a path from $i$ to $j$ in the graph $\mathcal{G}$. If a vertex $j\in \mathcal{V}$ is reachable from a vertex $i\in \mathcal{V}$, the distance from $i$ to $j$, $\text{dist}(i,j)$, is the length of the shortest path from $i$ to $j$. 
\end{defn}
\begin{oss}
	If the graph $\mathcal{G}$ has no self-loops, namely $(i,i)\notin \mathcal{E}$, for all $i=1,\dots,N$, then $i \notin \mathcal{N}_i$ and $0\leq N_i\leq N-1$, for all $1 \leq i \leq N$. 
\end{oss}
\begin{defn}[Strongly connected graph]
	A direct graph $\mathcal{G}=(\mathcal{V},\mathcal{E})$ is said to be strongly connected if, for any $i,j\in\mathcal{V}$ with $i\neq j$, the vertex $j$ is reachable from vertex $i$. In this case, we define the \emph{depth} of the direct graph $\mathcal{G}$ as
	\begin{equation}\label{depth}
		\gamma:=\max_{i,j=1,\dots,N} \text{dist}(i,j).
	\end{equation}
\end{defn}
\begin{defn}[Rooted graph]\label{rooted}
	A direct graph $\mathcal{G}=(\mathcal{V},\mathcal{E})$ is said to be rooted if there exists a vertex $r\in \mathcal{V}$ such that, for all $i\in \mathcal{V}$ with $i\neq r$, the vertex $i$ is reachable from vertex $r$. The distance from the root is 
	\begin{equation}\label{gamma}
		\gamma_r:=\max_{i =1,\dots, N} \text{dist}(r,i). 
	\end{equation}
\end{defn}
\begin{ex}
	Every strongly connected direct graph $\mathcal{G}$ is rooted. Other examples of rooted graphs are spanning trees. 
\end{ex}
\begin{oss}
	If a direct graph $\mathcal{G}$ is rooted, then the distance from the root satisfies $1\leq \gamma_r\leq N-1$.
\end{oss}

\vspace{0.3cm}

Now, we introduce a generalization of the classical Persistence Excitation Condition (cfr. \cite{Bonnet,AnconaRossi,CCP}). 
\begin{defn}[Generalized Persistence Excitation Condition]\label{GPE}
	Let $\{\alpha_{i,j}\}_{i,j=1,\dots,N}$, $N\in \mathbb{N}$, be a family of $\mathcal{L}^\infty([0,+\infty);[0,1])$ functions. We say that the family $\{\alpha_{i,j}\}_{i,j=1,\dots,N}$ satisfies a generalized Persistence Excitation Condition if there exist an exponent $\beta\geq 0$ and positive constants $T$ and $\tilde{\alpha}$ such that, for all $i,j=1,\dots,N$ such that $\chi_{ij}=1,$
	\begin{equation}\label{PE}
		\int_{t}^{t+T}\alpha_{ij}(s)ds\geq \frac{\tilde{\alpha}}{(1+t)^\beta},\quad \forall t\geq 0.
	\end{equation}
\end{defn}
\begin{oss}
	Let us note that, if the exponent $\beta$ in \eqref{PE} is 0, then Definition \ref{GPE} reduces to the definition of the classical Persistence Excitation Condition \eqref{classicalPE}.
\end{oss}
\begin{oss}
	Note that if the functions $\alpha_{ij}$ satisfy $\alpha_{ij}(t)=1$, for all $t\geq 0$, then \eqref{PE} is of course satisfied for $\beta=0$ and for any $T$ and $\tilde{\alpha}$ with $T\geq \tilde{\alpha}$.
\end{oss}

	\section{The first-order alignment model}\label{timedelay}
	\setcounter{equation}{0}
	Consider a finite set of $N\in \mathbb{N}$ interacting agents, $N\geq 2 $. The notation $x_{i}(t)\in \mathbb{R}^d$ will denote the position of the $i$-th agent at time $t$. One may think of $x_{i}(t)$ as a more abstract quantity: for instance, in the context of social sciences, $x_{i}(t)$ could be interpreted as the opinion of the $i$-th agent at time $t$. In the sequel, we shall denote with $\lvert\cdot \rvert$ and $\langle\cdot,\cdot\rangle$ the usual norm and scalar product on $\mathbb{R}^{d}$, respectively. Moreover, we set $\mathbb{N}_0:=\mathbb{N}\cup\{0\}.$ 
	
	\vspace{0.1cm}
	
	We assume that the agents interact among themselves according the following Hegselmann-Krause type law:
	\begin{equation}\label{onoff}
		\frac{d}{dt}x_{i}(t)=\underset{j:j\neq i}{\sum}\chi_{ij} b_{ij}(t)(x_{j}(t-\tau_{ij}(t))-x_{i}(t)),\quad t>0,	\,\, \forall i=1,\dots,N.
	\end{equation}
	
	Here, $\tau_{ij}:[0,+\infty)\rightarrow[0,+\infty)$ are the time delay functions. On the time delay functions we require the following classical assumption: for all $i,j=1,\dots,N$, $\tau_{ij}(\cdot)$ is continuous and bounded. Namely, there is a positive constant $\tau>0$ such that
	\begin{equation}\label{taubounded}
		0\leq \tau_{ij}(t)\leq \tau,\quad \forall t\geq 0,\,\forall i,j=1,\dots,N.
	\end{equation}
	 
    Due to the presence of the time delays, we prescribe initial conditions that are functions defined in the interval $[-\tau, 0]$, 
    \begin{equation}\label{incond}
    	x_{i}(s)=x^{0}_{i}(s),\quad \forall s\in [-\tau,0],\,\forall i=1,\dots,N.
    \end{equation}
    For every $i=1,\dots,N$, $x_i^0$ is taken in the space $C([-\tau,0];\mathbb{R}^d)$ of continuous functions from $[-\tau,0]$ to $\mathbb{R}^d$. 	
    
    \vspace{0.1cm}
    
	In \eqref{onoff}, the terms $b_{ij}$ are defined as follows,
	\begin{equation}\label{weight}
		b_{ij}(t):=\frac{1}{N-1}{\alpha_{ij}(t)}\psi( x_{i}(t), x_{j}(t-\tau_{ij}(t))), \quad t>0,\, \forall i,j=1,\dots,N.
	\end{equation}
	The function $\psi:\mathbb{R}^d\times\mathbb{R}^d\rightarrow \mathbb{R}$ in \eqref{weight} is called influence function and describes the interaction among the agents. Classically, the influence function is a function depending on the distance between agents' positions (or opinions), namely
	\begin{equation}\label{psi}
		\psi(x,y):=\tilde\psi(\lvert x-y\rvert),\quad \forall x,y\in\mathbb{R}^d,
	\end{equation}
	with $\tilde\psi:[0,+\infty)\rightarrow\mathbb{R}$  continuous, positive, and nonincreasing. The nonincreasing property of $\psi$ is actually a reasonable assumption since, in many physical models, as the distance between two particles grows, the interaction between them becomes weaker. However, in this section, the influence function is of more general type, namely $\psi$ is a generic function of agents' positions. Also, we weaken the monotonicity condition on $\psi$ by dealing with only bounded influence functions. Indeed, even if $\psi$ is nonmonotone, it suffices just to take a monotonization of $\psi$ in the proofs. In the sequel, we will denote
	\begin{equation}\label{K}
		K:=\lVert \psi\rVert_{\infty}.
	\end{equation} 
	
	The weights $\alpha_{ij}:[0,+\infty)\rightarrow [0,1]$ in \eqref{weight} describe possible communication failures among the agents since they are allowed to assume the value 0 at certain times, in which case the agents interrupt their exchange of information. In case of communication failures, consensus formation is not guaranteed in general, unless one requires some conditions on the weights ensuring that the interaction is not too weak. Most of the consensus results for multiagent systems with lack of interaction have been obtained under the Persistence Excitation Condition (PE) (cfr. \cite{AnconaRossi, Bonnet, CCP}). Here, we establish consensus formation under the Generalized Persistence Excitation Condition (see Definition \ref{GPE}). 
	
	\vspace{0.1cm}
	
	Finally, the terms $\chi_{ij}$ in \eqref{onoff} are defined as follows
    \begin{equation}\label{chiij}
    	{\chi_{ij}=}\begin{cases}
    		1,\quad \text{if } j \text{ transmits information to } i,\\
    		0,\quad \text{otherwise}.
    	\end{cases}
    \end{equation}
    When $\chi_{ij}=1$, for all $i,j=1,\dots,N$, each pair of agents can exchange information and the interaction is universal. On the other hand, if $\chi_{ij}=0$, for some $i,j=1,\dots,N$, the interaction is non-universal and there is at least a pair of agents $i$ and $j$ such that agent $j$ cannot influence agent $i$ at any times. 
    
    In order to deal with the non-universal interaction, we locate the agents at the vertices of a direct graph $\mathcal{G}=(\mathcal{V},\mathcal{E})$, $\mathcal{V}=\{1,\dots,N\}$, without self-loops and whose adjacency matrix is $\{\chi_{ij}\}_{i,j=1,\dots,N}$. In this paper, we show that a sufficient condition for consensus formation for system \eqref{onoff} is that the direct graph $\mathcal{G}$ is rooted (see Definition \ref{rooted}).

	\vspace{0.1cm}

Existence of solutions to \eqref{onoff} can be obtained using classical arguments for delay differential equations (cf. \cite{HL, Halanay}). Here, we will analyze the asymptotic behavior of solutions to \eqref{onoff}. In particular, we will establish asymptotic consensus (see Definition \ref{cons} below) for system \eqref{onoff}.

\begin{defn}[Consensus]\label{cons}
	Let $\{x_{i}\}_{i=1,\dots,N}$ be a solution to \eqref{onoff}. We define the diameter $d(\cdot)$ of the solution as follows,
	\begin{equation}\label{diamhk}
		d(t):=\max_{i,j=1,\dots,N}\lvert x_{i}(t)-x_{j}(t)\rvert,\quad \forall t\geq -\tau.
	\end{equation} 
	We say that the solution $\{x_{i}\}_{i=1,\dots,N}$ to \eqref{onoff} converges to consensus if $d(t)\to0$, as $t\to \infty$.
\end{defn}
Our main results is the following. 

\begin{thm}\label{consres}
	Assume that the direct graph $\mathcal{G}$ is rooted (see Definition \ref{rooted}). Assume \eqref{taubounded}. Let $\psi:\mathbb{R}^d\times\mathbb{R}^d\rightarrow\mathbb{R}$ be positive, bounded, and continuous. Assume the weights $\{\alpha_{ij}\}_{i,j=1,\dots,N}$ belong to $\mathcal{L}^\infty([0,+\infty);[0,1])$ and satisfy \eqref{PE} with exponent $0\leq \beta\leq \frac{1}{\gamma_{r}}$, where $\gamma_{r}$ is the distance from the root defined in \eqref{gamma}. Let $x_i^0$ be continuous in $[-\tau,0]$, for all $i=1,\dots,N$. Then, every solution $\{x_i\}_{i=1,\dots,N}$ to \eqref{onoff}, \eqref{incond} converges to consensus in the sense of Definition \ref{cons}.
\end{thm}	
\begin{oss}
	Note that, without loss of generality, we can assume that the positive constant in \eqref{PE} satisfies $\tilde{\alpha}\leq \frac{1}{K}$. Moreover, whereas the exponent $\beta$ in the generalized Persistence Excitation Condition \eqref{PE} is required to satisfy the constraint $0\leq \beta \leq \frac{1}{\gamma_{r}}$, no restrictions are required on the positive constant $T$. 
\end{oss}
	\subsection{Proof of the consensus result}
	 Let us consider a solution $\{x_{i}\}_{i=1,\dots,N}$ to \eqref{onoff}, \eqref{incond}. Throughout this subsection, we assume the validity of the assumptions of Theorem \ref{consres}. To prove the consensus result Theorem \ref{consres}, we first derive some preliminary estimates. We start by introducing some quantities.
    
		\begin{defn}\label{quantonoff}
		Given a vector $v\in \mathbb{R}^d$, for all $n\in \mathbb{N}_0$ we define
		\begin{equation}\label{In}
			I_n:=[n(\gamma_r (T+\tau)+\tau)-\tau,n(\gamma_r (T+\tau)+\tau)],
		\end{equation}
		\begin{equation}\label{m_n}
			m_n^v:=\min_{i=1,\dots,N}\min_{s\in I_n}\,\langle x_{i}(s),v\rangle,
		\end{equation}
		\begin{equation}\label{M_n}
			M_n^v:=\max_{j=1,\dots,N}\max_{s\in I_n}\,\langle x_{j}(s),v\rangle.
		\end{equation}
	\end{defn} 
    \begin{defn}[Generalized diameters]\label{Dn}
    	We define the generalized initial diameter
    	\begin{equation}\label{D0}
    		D_0:=\max_{i,j =1,\dots, N}\max_{\sigma,s\in I_0}\lvert x_i(\sigma)-x_j(s)\rvert=\max_{i,j =1,\dots, N}\max_{\sigma,s\in [-\tau,0]}\lvert x_i(\sigma)-x_j(s)\rvert.
    	\end{equation}
    	For all $n\in \mathbb{N}_0$, we define the $n$-generalized diameter
    	\begin{equation}\label{D_n}
    		D_n:=\max_{i,j =1,\dots, N}\max_{\sigma,s\in I_n}\lvert x_i(\sigma)-x_j(s)\rvert.
    	\end{equation}
    \end{defn}
      We now state the preliminary results. These results can be proven with analogous arguments to the ones used in \cite{CCP}, therefore we omit their proofs. 
    
	\begin{lem}\label{L1}{(See \cite{CCP}, {Lemma 2.3 and Lemma 2.4})}
		For each $n\in \mathbb{N}_0$ and for any vector $v\in \RR^{d}$, we have 
		\begin{equation}\label{scalpronoff}
			m_n^v\leq \langle x_{i}(t),v\rangle \leq M_n^v,
		\end{equation}for all  $t\geq n(\gamma_r(T+\tau)+\tau)-\tau$ and for any $i=1,\dots,N$, where $m_n^v$ and $M_n^v$ are defined in \eqref{m_n} and \eqref{M_n}, respectively, $\gamma_r$ is the distance from the root defined in \eqref{gamma} and $T$ is the positive constant in \eqref{PE}.
	\end{lem}

	\begin{lem}\label{lemmadiamonoff}{(See \cite{CCP}, {Lemma 2.5 and Remark 2.6})}
		For each $n\in \mathbb{N}_0$, we have 
		\begin{equation}\label{diamonoff}
			\lvert x_i(s)-x_j(t)\rvert\leq D_n,
		\end{equation}
		for all  $s,t\geq n(\gamma_r (T+\tau)+\tau)-\tau$ and for any $i,j=1,\dots,N$, where $D_n$ is the $n$-generalized diameter defined in \eqref{D_n}, $\gamma_r$ is the distance from the root defined in \eqref{gamma} and $T$ is the positive constant in \eqref{PE}.
		
		In particular, for all $n\in \mathbb{N}_0$, 	
		\begin{equation}\label{Dndec}
			D_{n+1}\leq D_n,
		\end{equation}
		and
	\begin{equation}\label{diamonoff2}
		d(t)\leq D_n,\quad \forall t\geq n(\gamma_r (T+\tau)+\tau)-\tau,
	\end{equation}
	where $d(\cdot)$ is the diameter of the solution defined in \eqref{diamhk}
	\end{lem}
Also, we are able to show that the trajectories are bounded, uniformly with respect to $t$ and $i=1,\dots,N$. This allows us to derive a positive bound from below on the communication rates, which is crucial in order to establish the asymptotic consensus.
	\begin{lem}\label{L3onoff}{(See \cite{CCP}, {Lemma 2.7})}
		We have \begin{equation}\label{boundsolonoff}
			\lvert x_{i}(t)\rvert\leq C_{0}:=\max_{j=1,\dots,N}\,\,\max_{s\in [-\tau, 0]}\lvert x_{j}(s)\rvert,
		\end{equation}
        for all $t\geq -\tau$ and $i=1,\dots,N$. 
        
        In particular,        \begin{equation}\label{stima_psionoff}
			\psi (x_i(t), x_j(t-\tau_{ij}(t)))\ge \psi_{0}:=\min_{\vert y\vert, \vert z\vert \le  C_{0}}\psi(y,z)>0,
		\end{equation}
		for all $t\ge 0$ and $i,j=1,\dots, N$.
	\end{lem}

Finally, before we move to the proof of the consensus result Theorem \ref{consres}, we establish the following crucial result, inspired by arguments in \cite{CCP,H3}.
	\begin{prop}\label{lemma3onoff}
		For all $n\in \mathbb{N}_0$ and $v\in \mathbb{R}^d$, it holds
		\begin{equation} \label{Bonoff}
			m_{n}^v+\Gamma_n(\langle x_r(n(\gamma_r(T+\tau)+\tau)),v \rangle-m^v_{n}) \leq \langle x_i(t),v\rangle\leq M^v_{n}-\Gamma_n(M^v_{n}-\langle x_r(n(\gamma_r(T+\tau)+\tau)),v \rangle), 
		\end{equation}
		for all $t\in I_{n+1}$ and $i =1,\dots, N$, where $I_{n+1}$ is the time interval defined in \eqref{In} and		\begin{equation}\label{Gammaonoff}
			\Gamma_n:=e^{-K (\frac{1}{2}(\gamma_{r}^2+3\gamma_r)(T+\tau)+\tau)}\left(\frac{\psi_{0}\tilde{\alpha}}{(N-1)(1+n(\gamma_r(T+\tau)+\tau)+(\gamma_r-1)T+\gamma_r\tau)^\beta}\right)^{\gamma_r},
		\end{equation}
		being $m_n^v$ and $M_n^v$ the quantities defined in \eqref{m_n} and \eqref{M_n}, respectively, $\gamma_r$ the distance from the root defined in \eqref{gamma}, $\beta$, $T$ and $\tilde{\alpha}$ the constants in \eqref{PE}.
	\end{prop}
	\begin{oss}
		Note that, due to $\tilde{\alpha}\leq \frac{1}{K}$, by definition of $\psi_{0}$ we have $\psi_{0}\tilde{\alpha}\leq 1$, from which $\Gamma_n\in (0,1)$.
	\end{oss}
	\begin{proof}[Proof of Proposition \ref{lemma3onoff}]
		 Let $n\in \mathbb{N}_0$ and $v \in \mathbb{R}^d$. Let us fix an index $i=1,\dots,N$ with $i\neq r$. Due to the fact that the direct graph $\mathcal{G}$ is rooted, we can always find a path from $r$ to $i$. Let $i_0=r,i_1,\dots,i_p=i$, $1\leq p\leq \gamma_r$, be the shortest from $r$ to $i$. Then, for a.e. $t\in [n(\gamma_r(T+\tau)+\tau),(n+1)(\gamma_r(T+\tau)+\tau)]$, from \eqref{scalpronoff} we can write
		$$\begin{array}{l}
			\vspace{0.3cm}\displaystyle{\frac{d}{dt}\langle x_r(t),v\rangle=\sum_{j:j\neq r}\chi_{rj}(t)b_{rj}(t)(\langle x_j(t-\tau_{rj}(t)),v\rangle-\langle x_r(t),v\rangle)}\\
			\vspace{0.3cm}\displaystyle{\hspace{1.5cm}\leq \sum_{j:j\neq r}\chi_{rj}b_{rj}(t)(M_n^v-\langle x_r(t),v\rangle)}\\
			\displaystyle{\hspace{1.5cm}\leq \frac{K}{N-1}\sum_{j:j\neq r}(M_n^v-\langle x_r(t),v\rangle)=K(M_n^v-\langle x_r(t),v\rangle).}
		\end{array}$$
		The Gronwall's inequality yields
		\begin{equation}\label{primopasso}\begin{array}{l}
			\vspace{0.3cm}\displaystyle{\langle x_r(t),v\rangle\leq e^{-K(t-n(\gamma_r(T+\tau)+\tau))}\langle x_r(n(\gamma_r(T+\tau)+\tau)),v\rangle+M_n^v(1-e^{-K(t-n(\gamma_r(T+\tau)+\tau)})}\\
			\vspace{0.3cm}\displaystyle{\hspace{1.5cm}=M_n^v-e^{-K(t-n(\gamma_r(T+\tau)+\tau))}(M_n^v-\langle x_r(n(\gamma_r(T+\tau)+\tau)),v \rangle )}\\
			\displaystyle{\hspace{1.5cm}\leq M_n^v-e^{-K(\gamma_r(T+\tau)+\tau)}(M_n^v-\langle x_r(n(\gamma_r(T+\tau)+\tau)),v \rangle),}
		\end{array}
        \end{equation}
        for all $t\in [n(\gamma_r(T+\tau)+\tau),(n+1)(\gamma_r(T+\tau)+\tau)]$.
		\\Now, for a.e. $t \in [n(\gamma_r(T+\tau)+\tau)+\tau,(n+1)(\gamma_r(T+\tau)+\tau)]$, using \eqref{scalpronoff} and \eqref{primopasso} we get
		$$\begin{array}{l}
			\vspace{0.3cm}\displaystyle{\frac{d}{d t}  \langle x_{i_1}(t),v\rangle =  \sum_{j \neq i_1, r} \chi_{i_1j} b_{i_1j}(t)( \langle x_{j}(t-\tau_{i_1j}(t)),v\rangle-\langle x_{i_1}(t),v\rangle)+ b_{i_1r}(t)(\langle x_{r}(t-\tau_{i_1r}(t)),v\rangle-\langle x_{i_1}(t),v\rangle) }\\
			\vspace{0.3cm}\displaystyle{\hspace{1cm}\leq \sum_{j \neq i_1 ,r} \chi_{i_1j} b_{i_1j}(t)( M_{n}^v-\langle x_{i_1}(t),v\rangle)}\\
			\vspace{0.3cm}\displaystyle{\hspace{3cm}+b_{i_1r}(t)\left(M_n^v-e^{-K(\gamma_r(T+\tau)+\tau)}(M_n^v-\langle x_r(n(\gamma_r(T+\tau)+\tau)),v \rangle)-\langle x_{i_1}(t),v\rangle\right)}\\
			\vspace{0.3cm}\displaystyle{\hspace{1cm}=( M_{n}^v-\langle x_{i_1}(t),v\rangle)\sum_{j \neq i_1 , r} \chi_{i_1j} b_{i_1j}(t)}\\
			\displaystyle{\hspace{2cm}+ b_{i_1 r}(t)\left(M_n^v-e^{-K(\gamma_r(T+\tau)+\tau)}(M_n^v-\langle x_r(n(\gamma_r(T+\tau)+\tau)),v \rangle)-\langle x_{i_1}(t),v\rangle\right).}
		\end{array}$$
		Noticing that
		$$\sum_{j \neq i_1 ,r} \chi_{i_1j} b_{i_1j}(t)=\sum_{\substack{j \neq i_1 }} \chi_{i_1j} b_{i_1j}(t)-b_{i_1r}(t)\leq \frac{KN_{i_1}}{N-1}-b_{i_1r}(t),$$
		from \eqref{stima_psionoff} we obtain
		$$\begin{array}{l}
			\vspace{0.3cm}\displaystyle{\frac{d}{d t}  \langle x_{i_1}(t),v\rangle \leq \frac{KN_{i_{1}}}{N-1}(M_{n}^v-\langle x_{i_1}(t),v\rangle) -b_{i_1r}(t)(M_{n}^v-\langle x_{i_1}(t),v\rangle) }\\
			\vspace{0.3cm}\displaystyle{\hspace{3cm}+ b_{i_1r}(t)\left(M_n^v-e^{-K(\gamma_r(T+\tau)+\tau)}(M_n^v-\langle x_r(n(\gamma_r(T+\tau)+\tau)),v \rangle)-\langle x_{i_1}(t),v\rangle\right)}\\
			\vspace{0.3cm}\displaystyle{\hspace{1cm}=\frac{KN_{i_{1}}}{N-1}(M_{n}^v-\langle x_{i_1}(t),v\rangle) -e^{-K(\gamma_r(T+\tau)+\tau)}(M_n^v-\langle x_r(n(\gamma_r(T+\tau)+\tau)),v \rangle)b_{i_1r}(t)}\\
			\vspace{0.3cm}\displaystyle{\hspace{1cm}\leq \frac{KN_{i_{1}}}{N-1}(M_{n}^v-\langle x_{i_1}(t),v\rangle) -e^{-K(\gamma_r(T+\tau)+\tau)}(M_n^v-\langle x_r(n(\gamma_r(T+\tau)+\tau)),v \rangle)\alpha_{i_1r}(t)\frac{\psi_{0}}{N-1}}\\
			\displaystyle{\hspace{1cm}=\frac{KN_{i_{1}}}{N-1} M_n^v-e^{-K(\gamma_r(T+\tau)+\tau)}(M_n^v-\langle x_r(n(\gamma_r(T+\tau)+\tau)),v \rangle)\alpha_{i_1r}(t)\frac{\psi_{0}}{N-1}-\frac{KN_{i_{1}}}{N-1}\langle x_{i_1}(t),v\rangle.}
		\end{array}$$
		Hence, the Gronwall's inequality gives
		$$ \begin{array}{l}
			\vspace{0.3cm}\displaystyle{\langle x_{i_1}(t),v\rangle\leq e^{-\frac{KN_{i_1}}{N-1}(t-n(\gamma_r(T+\tau)+\tau)-\tau)} \langle x_{i_1}(n(\gamma_r(T+\tau)+\tau)+\tau),v\rangle+M^v_n(1-e^{-\frac{KN_{i_1}}{N-1}(t-n(\gamma_r(T+\tau)+\tau)-\tau)})} \\
			\vspace{0.3cm}\displaystyle{\hspace{1.5cm}-e^{-K(\gamma_r(T+\tau)+\tau)}(M_n^v-\langle x_r(n(\gamma_r(T+\tau)+\tau)),v \rangle)\frac{\psi_0}{N-1}\int_{n(\gamma_r(T+\tau)+\tau)+\tau}^{t}\alpha_{i_1r}(s)e^{-\frac{KN_{i_1}}{N-1}(t-s)}ds}\\
			\vspace{0.3cm}\displaystyle{\hspace{0.8cm}\leq e^{-\frac{KN_{i_1}}{N-1}(t-n(\gamma_r(T+\tau)+\tau)-\tau)}M_n^v +M^v_n(1-e^{-\frac{KN_{i_1}}{N-1}(t-n(\gamma_r(T+\tau)+\tau)-\tau)})}\\
			\vspace{0.3cm}\displaystyle{\hspace{1.5cm}-e^{-K(\gamma_r(T+\tau)+\tau)}(M_n^v-\langle x_r(n(\gamma_r(T+\tau)+\tau)),v \rangle)e^{-K\gamma_r(T+\tau)}\frac{\psi_0}{N-1}\int_{n(\gamma_r(T+\tau)+\tau)+\tau}^{t}\alpha_{i_1r}(s)ds}\\
			\vspace{0.3cm}\displaystyle{\hspace{0.8cm}=M^v_n-e^{-K(2\gamma_r(T+\tau)+\tau)}(M_n^v-\langle x_r(n(\gamma_r(T+\tau)+\tau)),v \rangle)\frac{\psi_0}{N-1}\int_{n(\gamma_r(T+\tau)+\tau)}^{t}\alpha_{i_1r}(s)ds,}
		\end{array}$$
		for all $t\in [n(\gamma_r(T+\tau)+\tau)+\tau,(n+1)(\gamma_r(T+\tau)+\tau)]$. In particular, for $t\in [n(\gamma_r(T+\tau)+\tau)+T+\tau,(n+1)(\gamma_r(T+\tau)+\tau)]$, since the Generalized Persistence Excitation Condition \eqref{PE} implies that 
		$$\int_{n(\gamma_r(T+\tau)+\tau)}^{t}\alpha_{i_1r}(s)ds\geq \int_{n(\gamma_r(T+\tau)+\tau)+\tau}^{n(\gamma_r(T+\tau)+\tau)+T+\tau}\alpha_{i_1r}(s)ds\geq \frac{\tilde{\alpha}}{(1+n(\gamma_r(T+\tau)+\tau)+\tau)^\beta},$$
		we obtain
		\begin{equation}\label{i_1T}
			\langle x_{i_1}(t),v\rangle\leq M^v_n-e^{-K(2\gamma_r(T+\tau)+\tau)}(M_n^v-\langle x_r(n(\gamma_r(T+\tau)+\tau)),v \rangle)\frac{\psi_0\tilde{\alpha}}{(N-1)(1+n(\gamma_r(T+\tau)+\tau)+\tau)^\beta}.
		\end{equation}	
		At this point, if $p=1$, we stop the argument. Otherwise, if $p>1$, using \eqref{i_1T} we estimate, for $t\in [n(\gamma_r(T+\tau)+\tau)+T+2\tau,(n+1)(\gamma_r(T+\tau)+\tau)]$,
		$$\begin{array}{l}
			\vspace{0.3cm}\displaystyle{\frac{d}{d t}  \langle x_{i_2}(t),v\rangle =  \sum_{j \neq i_1, i_2} \chi_{i_2j} b_{i_2j}(t)( \langle x_{j}(t-\tau_{i_2j}(t)),v\rangle-\langle x_{i_2}(t),v\rangle)+ b_{i_2i_1}(t)(\langle x_{i_1}(t-\tau_{i_2i_1}(t)),v\rangle-\langle x_{i_2}(t),v\rangle) }\\
			\vspace{0.3cm}\displaystyle{\hspace{2cm}\leq ( M_{n}^v-\langle x_{i_2}(t),v\rangle)\sum_{j \neq i_1, i_1} \chi_{i_2j} b_{i_2j}(t)} \\
			\displaystyle{+ b_{i_2i_1}(t)\Big (M^v_n-e^{-K(2\gamma_r(T+\tau)+\tau)}(M_n^v-\langle x_r(n(\gamma_r(T+\tau)+\tau)),v \rangle)\frac{\psi_0\tilde{\alpha}}{(N-1)(1+n(\gamma_r(T+\tau)+\tau)+\tau)^\beta}}\\ \vspace{0.3cm}
				\displaystyle{\hspace{16 cm} -\langle x_{i_2}(t),v\rangle\Big).}
		\end{array}$$
		Therefore, arguing as above, 
		$$\begin{array}{l}
			\vspace{0.3cm}\displaystyle{\frac{d}{d t}  \langle x_{i_2}(t),v\rangle \leq \frac{KN_{i_{2}}}{N-1}(M_{n}^v-\langle x_{i_2}(t),v\rangle) -b_{i_2i_1}(t)(M_{n}^v-\langle x_{i_2}(t),v\rangle) + b_{i_2i_1}(t)(M^v_n-\langle x_{i_2}(t),v\rangle)}\\
			\vspace{0.3cm}\displaystyle{\hspace{1.5cm}-b_{i_2i_1}(t)e^{-K(2\gamma_r(T+\tau)+\tau)}(M_n^v-\langle x_r(n(\gamma_r(T+\tau)+\tau)),v \rangle)\frac{\psi_0\tilde{\alpha}}{(N-1)(1+n(\gamma_r(T+\tau)+\tau)+\tau)^\beta}}\\
			\vspace{0.3cm}\displaystyle{\hspace{0.8cm}\leq \frac{KN_{i_{2}}}{N-1}M_{n}^v-\frac{KN_{i_{2}}}{N-1}\langle x_{i_2}(t),v\rangle}\\
            \displaystyle{\hspace{1.5cm}-\alpha_{i_2i_1}(t)e^{-K(2\gamma_r(T+\tau)+\tau)}(M_n^v-\langle x_r(n(\gamma_r(T+\tau)+\tau)),v \rangle)\left(\frac{\psi_0}{N-1}\right)^2\frac{\tilde{\alpha}}{(1+n(\gamma_r(T+\tau)+\tau)+\tau)^\beta}.}
		\end{array}$$
		Again, the Gronwall's inequality yields
		$$ \begin{array}{l}
			\vspace{0.3cm}\displaystyle{\langle x_{i_2}(t),v\rangle\leq M^v_n -e^{-K(2\gamma_r(T+\tau)+{\tau})}(M_n^v-\langle x_r(n(\gamma_r(T+\tau)+\tau)),v \rangle)\left(\frac{\psi_0}{N-1}\right)^2\frac{\tilde{\alpha}}{(1+n(\gamma_r(T+\tau)+\tau)+\tau)^\beta}\times}\\
			\vspace{0.3cm}\displaystyle{\hspace{4cm} \times\int_{n(\gamma_r(T+\tau)+\tau)+T+2\tau}^{t}\alpha_{i_2i_1}(s)e^{-\frac{KN_{i_2}}{N-1}(t-s)}ds}\\
			\vspace{0.3cm}\displaystyle{\hspace{1.5cm}\leq M^v_n -e^{-K(3\gamma_r(T+\tau)-T)}(M_n^v-\langle x_r(n(\gamma_r(T+\tau)+\tau)),v \rangle)\left(\frac{\psi_0}{N-1}\right)^2\frac{\tilde{\alpha}}{(1+n(\gamma_r(T+\tau)+\tau)+\tau)^\beta}\times}\\
			\displaystyle{\hspace{4cm}\times\int_{n(\gamma_r(T+\tau)+\tau)+T+2\tau}^{t}\alpha_{i_2i_1}(s)ds,}
		\end{array}$$
		for all $t\in [n(\gamma_r(T+\tau)+\tau)+T+2\tau,T+2\tau,(n+1)(\gamma_r (T+\tau)+\tau)]$. In particular, for $t\in [n(\gamma_r(T+\tau)+\tau)+2T+2\tau,(n+1)(\gamma_r(T+\tau)+\tau)]$, since from \eqref{PE} it comes that
		$$\int_{n(\gamma_r(T+\tau)+\tau)+T+2\tau}^{t}\alpha_{i_2i_1}(s)ds\geq \int_{n(\gamma_r(T+\tau)+\tau)+T+2\tau}^{n(\gamma_r(T+\tau)+\tau)+2T+2\tau}\alpha_{i_2i_1}(s)ds\geq \frac{\tilde\alpha}{(1+n(\gamma_r(T+\tau)+\tau)+T+2\tau)^\beta},$$
		we can conclude that
		\begin{equation}\label{i_2T}
			\begin{array}{l}
				\vspace{0.3cm}\displaystyle{\langle x_{i_2}(t),v\rangle\leq M^v_n-e^{-K(3\gamma_r (T+\tau)-T)}(M_n^v-\langle x_r(n(\gamma_r(T+\tau)+\tau)),v \rangle)\left(\frac{\psi_0\tilde{\alpha}}{N-1}\right)^2\times}\\
				\vspace{0.3cm}\displaystyle{\hspace{1.8cm}\times\frac{1}{(1+n(\gamma_r(T+\tau)+\tau)+T+2\tau)^\beta}\frac{1}{(1+n(\gamma_r(T+\tau)+\tau)+\tau)^\beta}}\\
				\displaystyle{\hspace{0.8cm}\leq M^v_n-e^{-K(3\gamma_r (T+\tau)-T)}(M_n^v-\langle x_r(n(\gamma_r(T+\tau)+\tau)),v \rangle)\left(\frac{\psi_0\tilde{\alpha}}{(N-1)(1+n(\gamma_r(T+\tau)+\tau)+T+2\tau)^\beta}\right)^2.}
			\end{array}
		\end{equation}
		Finally, iterating the above procedure,
		\begin{equation} \label{5.13onoff}
			\begin{array}{l}
				\vspace{0.3cm}\displaystyle{\langle x_{i_k}(t),v\rangle \leq M^v_{n}-e^{- K((k+1)\gamma_r(T+\tau) -\left(\sum_{l=0}^{k-1}l\right)(T+\tau)+\tau)} (M_n^v-\langle x_r(n(\gamma_r(T+\tau)+\tau)),v \rangle)\times}\\
				\displaystyle{\hspace{2.5cm}\times\left(\frac{\psi_{0}\tilde{\alpha}}{(N-1)(1+n(\gamma_r(T+\tau)+\tau)+(k-1)T+k\tau)^\beta}\right)^{k},}
			\end{array}
		\end{equation}
		for all $1\leq k\leq p$ and for all $t\in [n(\gamma_r (T+\tau)+\tau)+k(T+\tau),(n+1)(\gamma_r (T+\tau)+\tau)]$. In particular, for $k=p$, inequality \eqref{5.13onoff} reads as
        \begin{equation}\label{agenti}
        \begin{split}
           & \langle x_{i}(t),v\rangle \leq M^v_{n}-e^{- K((p+1)\gamma_r(T+\tau) -\left(\sum_{l=0}^{p-1}l\right)(T+\tau)+\tau)} (M_n^v-\langle x_r(n(\gamma_r(T+\tau)+\tau)),v \rangle) \\
           & \hspace{4cm}\left(\frac{\psi_{0}\tilde{\alpha}}{(N-1)(1+n(\gamma_r(T+\tau)+\tau)+(p-1)T+p\tau)^\beta}\right)^{p},
            \end{split}
        \end{equation} 
        for all $t\in [n(\gamma_r (T+\tau)+\tau)+p(T+\tau),(n+1)(\gamma_r (T+\tau)+\tau)]$.
        \\Note that, if the path has length $p=\gamma_r$, using that $\sum_{l=0}^{\gamma_r-1}l=\frac{\gamma_r(\gamma_r-1)}{2}$, estimate \eqref{agenti} gives
		\begin{equation}\label{5.13gammaonoff}
        \begin{array}{l}
        	\vspace{0.3cm}\displaystyle{\langle x_{i}(t),v\rangle \leq M^v_{n}}\\
        	\displaystyle{\hspace{0.1cm}-e^{-K (\frac{1}{2}(\gamma_{r}^2+3\gamma_r)(T+\tau)+\tau)}(M_n^v-\langle x_r(n(\gamma_r(T+\tau)+\tau)),v \rangle)\Big(\frac{\psi_{0}\tilde{\alpha}}{(N-1)(1+n(\gamma_r(T+\tau)+\tau)+(\gamma_r-1)T+\gamma_r\tau)^\beta}\Big)^{\gamma_r},}
        \end{array}
		\end{equation}
		for all $t\in [(n+1)(\gamma_r(T+\tau)+\tau))-\tau,(n+1)(\gamma_r(T+\tau)+\tau))]=I_{n+1}$. Hence, due to the fact that \eqref{5.13gammaonoff} is the coarsest estimate, we conclude that \eqref{5.13gammaonoff} holds for every $i=1,\dots,N$ and $t\in I_{n+1}$. 
\\On the other hand, using analogous arguments one can show that
		\begin{equation*}
		    \begin{array}{l}
		    	\vspace{0.3cm}\displaystyle{\langle x_{i}(t),v\rangle \geq m^v_{n}}\\
		    	\displaystyle{\hspace{0.1cm}+e^{-K(\frac{1}{2}(\gamma_{r}^2+3\gamma_r)(T+\tau)+\tau)}(\langle x_r(n(\gamma_r(T+\tau)+\tau)),v \rangle-m_{n}^v)\Big(\frac{\psi_{0}\tilde{\alpha}}{(N-1)(1+n(\gamma_r(T+\tau)+\tau)+(\gamma_r-1)T+\gamma_r\tau)^\beta}\Big )^{\gamma_r}.}
		    \end{array}
		\end{equation*}
		for all $i=1,\dots,N$ and $t\in I_{n+1}$. Combining this last estimate with \eqref{5.13gammaonoff} we finally obtain \eqref{Bonoff}.
	\end{proof} 
    We are now able to prove Theorem \ref{consres}.
	\begin{proof}[Proof of Theorem \ref{consres}]
		Let $v\in \RR^d$. Let us denote with $$\mathcal{D}^v_n:=M^v_n-m^v_n,\quad \forall n\in \mathbb{N}_0,$$
		where $m_n^v$, $M_n^v$ are defined in \eqref{m_n} and \eqref{M_n}, respectively. Note that, for all $n\in \mathbb{N}_0$, due to $M_n^v\geq m_n^v$ we have $\mathcal{D}^v_n\geq 0$.
		\\Let $\Gamma_n\in (0,1)$ be the constant defined in \eqref{Gammaonoff}. We claim that
		\begin{equation}\label{D_nonoff}
			\mathcal{D}^v_{n+1}\leq (1-\Gamma_n)\mathcal{D}_{n}^v,\quad \forall n\in \mathbb{N}_0.
		\end{equation}
		Indeed, given $n\in \mathbb{N}_0$, if $i,j=1,\dots,N$ and $s,t\in I_{n+1}$ are such that $\langle x_i(s),v\rangle=M^v_{n+1}$ and $\langle x_j(t),v\rangle=m^v_{n+1}$, applying estimate \eqref{Bonoff} we find 
		\begin{equation}\label{calD}
			\begin{array}{l}
            \vspace{0.3cm}\displaystyle{\hspace{1.5cm}\mathcal{D}^v_{n+1}=\langle x_i(s),v\rangle-\langle x_j(t),v\rangle}\\
				\vspace{0.3cm}\displaystyle{\hspace{2cm}\leq M_{n}^v-m_{n}^v-\Gamma_n(M_{n}^v-\langle x_r(n(\gamma_r(T+\tau)+\tau)),v \rangle)-\Gamma_n(\langle x_r(n(\gamma_r(T+\tau)+\tau)),v \rangle-m_{n}^v)} \\
                \displaystyle{\hspace{2cm}=M_{n}^v-m_{n}^v-\Gamma_n(M_{n}^v-m_{n}^v)=(1-\Gamma_n)\mathcal{D}_n^v.}                
			\end{array}
		\end{equation}
		Moreover, since the constant $\Gamma_n$ in \eqref{D_nonoff} is independent of the choice of the vector $v\in\mathbb{R}^d$, from \eqref{D_nonoff} we obtain
		\begin{equation}\label{deconoff}
			D_{n+1}\leq (1-\Gamma_n)D_{n},\quad \forall n\in \mathbb{N}_0.
		\end{equation}
	To see this, let $n\in \mathbb{N}_0$. Let $i,j=1,\dots,N$ and $s,t\in I_{n+1}$ be such that $D_{n+1}=\lvert x_i(s)-x_j(t)\rvert.$ We can assume without loss of generality that $D_{n+1}>0$. Then, defining the unit vector $v = \frac{x_i(s)-x_j(t)}{\lvert x_i(s)-x_j(t)\rvert}$, from \eqref{scalpronoff} and \eqref{D_nonoff} we find
		 $$\begin{array}{l}
			\vspace{0.3cm}\displaystyle{D_{n+1}=\langle x_i(s),v\rangle-\langle x_j(t),v\rangle\leq M^v_{n+1}-m^v_{n+1}=\mathcal{D}^v_{n+1}}\\
			\vspace{0.3cm}\displaystyle{\hspace{1cm}\leq (1-\Gamma_n)\mathcal{D}^v_{n}=(1-\Gamma_n)(M^v_{n}-m^v_{n})}\\
			\displaystyle{\hspace{1cm}\leq (1-\Gamma_n)\max_{k,l=1,\dots,N}\max_{r,w\in I_n}\lvert x_{k}(r)-x_{l}(w)\rvert=(1-\Gamma_n)D_{n}.}
		\end{array}$$
		Now, from \eqref{deconoff}, an induction argument gives
		\begin{equation*}
			D_{n}\leq \prod_{k=0}^{n-1}(1-\Gamma_k)D_0,\quad \forall n\in \mathbb{N},
		\end{equation*}
	  that can be rewritten equivalently as follows
		\begin{equation}\label{decayonoffriscritta}
			D_{n}\leq e^{{\sum\limits_{k=0}^{n-1}\ln(1-\Gamma_k)}}D_0,\quad \forall n\in \mathbb{N}.
		\end{equation}
		Moreover, using \eqref{diamonoff2}, we can write $$d(t)\leq e^{\sum\limits_{k=0}^{n-1}\ln(1-\Gamma_k)}D_0,\quad \forall t\geq n(\gamma_r(T+\tau)+\tau)-\tau,\,\forall n\in\mathbb{N}.$$
		Therefore, to prove the asymptotic consensus it just remains to show that 
        $\sum\limits_{k=0}^{\infty}\ln(1-\Gamma_k)=-\infty$. Notice that $\Gamma_k\to 0$, as $k\to \infty$. Thus, $\ln(1-\Gamma_k)\sim-\Gamma_k\sim -\frac{1}{k^{\beta\gamma_r}}$, as $k\to \infty$. Thus, using the assumption $\beta\leq \frac{1}{\gamma_r}$ we have $\sum\limits_{k=0}^{\infty}\ln(1-\Gamma_k)=-\infty$. This concludes the proof.
	\end{proof}
\begin{oss}[Case $\beta=0$]\label{casobeta0}
	Note that, if $\beta=0$ in the Generalized Persistence Excitation Condition \eqref{PE}, namely if the weights satisfy the Persistence Excitation Condition \eqref{classicalPE}, then consensus is achieved exponentially fast. Indeed, in this case (see the definition \eqref{Gammaonoff}) we have $\Gamma_n=e^{-K (\frac{1}{2}(\gamma_{r}^2+3\gamma_r)(T+\tau)+\tau)}\left(\frac{\psi_{0}\tilde{\alpha}}{N-1}\right)^{\gamma_r}:=\Gamma$, for all $n\in\mathbb{N}_0$. Therefore, \eqref{deconoff} reads as 
	$$D_{n+1}\leq (1-\Gamma)D_n.$$
	Arguing as in the proof of Theorem \ref{consres}, an induction argument gives
	$$D_n\leq e^{-n\ln\left(\frac{1}{1-\Gamma}\right)}D_0,\quad \forall n\in \mathbb{N}_0.$$
	At this point, given $t\geq 0$, since $t\in [n(\gamma_r(T+\tau)+\tau),(n+1)(\gamma_r(T+\tau)+\tau)]$, for some $n\in\mathbb{N}_0$, the above estimate together with \eqref{diamonoff2} gives 
	$$d(t)\leq D_n\leq e^{-n\ln\left(\frac{1}{1-\Gamma}\right)}D_0\leq e^{-\frac{1}{\gamma_r(T+\tau)+\tau}\ln\left(\frac{1}{1-\Gamma}\right)(t-\gamma_r(T+\tau)-\tau)}D_0.$$
	So, we have exponential decay of the diameter of the solution.
\end{oss}
\section{The second-order alignment\label{4} model}\label{secdelay}
\setcounter{equation}{0}
	In this section, we analyze the second-order version of system \eqref{onoff}. We have a finite set of $N\in\N$ particles, with $N\geq 2 $. We denote with $x_{i}(t)\in \RR^d$ and $v_{i}(t)\in \RR^d$ the position and the velocity of the $i$-th particle at time $t$, respectively. We assume that the agents interact among themselves according to the following Cucker-Smale type rule:
	\begin{equation}\label{csp}
		\begin{cases}
			\frac{d}{dt}x_{i}(t)=v_{i}(t),\quad &t>0, \,\,\forall i=1,\dots,N,\\\frac{d}{dt}v_{i}(t)=\underset{j:j\neq i}{\sum}\chi_{ij}b_{ij}(t)(v_{j}(t-\tau_{ij}(t))-v_{i}(t)),\quad 	&t>0,\,\,\forall i=1,\dots,N.
		\end{cases}
	\end{equation}

	As before, the time delay functions $\tau_{ij}:[0,+\infty)\rightarrow[0,+\infty)$ are continuous and satisfy \eqref{taubounded}, the communication rates $b_{ij}$ are defined in \eqref{weight} with an influence function of the form \eqref{psi}, namely depending on the distance between agents' positions. 
	
	We prescribe initial conditions that are continuous functions from $[-\tau,0]$ to $\mathbb{R}^d$
	\begin{equation}\label{incondcs}
		x_{i}(s)=x^{0}_{i}(s),\quad v_{i}(s)=v^{0}_{i}(s), \quad \forall s\in [-\tau,0],\,\forall i=1,\dots,N,
	\end{equation}
	where $\tau>0$ is the maximal value of the time delay sizes. 
	
    \vspace{0.1cm}

    As in the previous section, the terms $\chi_{ij}$ are defined as in \eqref{chiij} and we equip our model with a network topology. Namely, we consider the direct graph $\mathcal{G}=(\mathcal{V},\mathcal{E})$, $\mathcal{V}=\{1,\dots,N\}$, whose adjacency matrix is $\{\chi_{ij} \}_{i,j=1,\dots,N}$. 
    
    \vspace{0.1cm}
    
    For solutions to \eqref{csp}, we will establish the asymptotic flocking. 
    
    \begin{defn}[Unconditional flocking]\label{unflock}
    	Let $\{x_{i},v_i\}_{i=1,\dots,N}$ be a solution to \eqref{csp}. We define the position diameter $d_X(\cdot)$ and the velocity diameter $d_V(\cdot)$ of the solution as follows, 
    	\begin{equation}\label{posdiam}
    		d_{X}(t):=\max_{i,j=1,\dots,N}\lvert x_{i}(t)-x_{j}(t)\rvert,\quad \forall t\geq-\tau,
    	\end{equation}
    	\begin{equation}\label{veldiam}
    		d_{V}(t):=\max_{i,j=1,\dots,N}\lvert v_{i}(t)-v_{j}(t)\rvert,\quad \forall t\geq -\tau.
    	\end{equation}
    	We say that the solution $\{(x_{i},v_{i})\}_{i=1,\dots,N}$ to \eqref{csp} exhibits asymptotic flocking if the two following conditions are fulfilled:
    	\begin{enumerate}
    		\item there exists a positive constant $d^{*}$ such that $$\sup_{t\geq-{\tau}}d_{X}(t)\leq d^{*};$$
    		\item $d_V(t)\to 0$, as $t\to \infty$.
    	\end{enumerate}
    	
    \end{defn}
Our main result is the following.
\begin{thm} \label{uf}
	Assume that the direct graph $\mathcal{G}$ is rooted (see Definition \ref{rooted}) and \eqref{taubounded}. Assume the weights $\{\alpha_{ij}\}_{i,j=1,\dots,N}$ belong to $\mathcal{L}^\infty([0,+\infty);[0,1])$ and satisfy \eqref{PE} with exponent $0\leq\beta<\frac{1}{\gamma_r}$, where $\gamma_{r}$ is the distance from the root defined in \eqref{gamma}. Let $\psi:\mathbb{R}^d\times\mathbb{R}^d\rightarrow\mathbb{R}$ be as in \eqref{psi} and suppose that {$\tilde \psi$ is bounded, continuous, and satisfies 
    \begin{equation}\label{infint}
		\tilde\psi(t)\geq \frac{1}{(1+t)^\eta},\quad \forall t\geq 0,
	\end{equation}}
    where $\eta>0$ is such that $(\eta+\beta)\gamma_r<1$.
    Moreover, let $x^{0}_{i},v_{i}^{0}:[-{\tau},0]\rightarrow \RR^{d}$ be continuous functions, for any $i=1,\dots,N$. Then, every solution $\{(x_{i},v_{i})\}_{i=1,\dots,N}$ to \eqref{csp}, \eqref{incondcs}, exhibits asymptotic flocking in the sense of Definition \ref{unflock}.
	\end{thm}
\begin{oss}\label{condizioneintegrale}
Let us note that, due to $\beta\gamma_r\geq 0$, $\eta\gamma_r<1-\gamma_r\beta\leq 1$. Thus, condition \eqref{infint} implies
\begin{equation}\label{infint2}
    \int_{0}^{+\infty}\left(\min_{r\in [0,t]}\tilde\psi(r)\right)^{\gamma_r}dt=+\infty.
\end{equation}
	In particular, if the influence function $\tilde\psi$ is nonincreasing and $\gamma_r=1$, namely the root can directly influence all the other agents, the condition \eqref{infint2} reduces to 
    the classical sufficient condition for the unconditional flocking (see e.g. \cite{Ha1}):
    \begin{equation}\label{unflockcond}
    \int_{0}^{+\infty}\tilde\psi(t)dt=+\infty.
    \end{equation}
	Here, however, since the weights satisfy the Generalized Persistence Excitation Condition \eqref{PE} the integrability condition \eqref{infint2} is not sufficient to obtain the unconditional flocking {and we need to the stronger condition \eqref{infint}. This will be clear in the proof of Theorem \ref{uf}.} Only in the case $\beta=0$, the condition \eqref{infint2} is sufficient for the unconditional flocking (see Remark \ref{beta=0}).  

    Moreover, differently from the consensus result for the first-order model Theorem \ref{consres}, {we need to require that the exponent $\beta$ in the Generalized Persistence Excitation Condition satisfies the strict inequality $\beta<\frac{1}{\gamma_r}$.}
\end{oss}
\begin{oss}
    Without loss of generality, we can assume that the positive constant in the Generalized Persistence Excitation Condition \eqref{PE} satisfies $\tilde{\alpha}\leq \frac{1}{K}$.
\end{oss}
	\subsection{Proof of the flocking result}
	Let us consider a solution $\{x_{i},v_{i}\}_{i=1,\dots,N}$ to \eqref{csp}, \eqref{incondcs}. Throughout this subsection, we assume the validity of the assumptions of Theorem \ref{uf}. Using analogous arguments to the ones employed in the previous section, we have the following preliminary estimates. Before stating the preliminary results, we introduce the following definitions. 
	\begin{defn}\label{quantcs}
		Given a vector $v\in \mathbb{R}^d$, for all $n\in \mathbb{N}_0$ we define
		\begin{equation}\label{r_n}
			{r}_n^v:=\min_{j=1,\dots,N}\min_{s\in I_n}\,\langle v_{j}(s),v\rangle,
		\end{equation}
		\begin{equation}\label{R_n}
			{R}_n^v:=\max_{j=1,\dots,N}\max_{s\in I_n}\,\langle v_{j}(s),v\rangle,
		\end{equation}	
where $I_n$ is the time interval defined in \eqref{In} of Definition \ref{quantonoff}, namely
$I_n=[n(\gamma_r (T+\tau)+\tau)-\tau,n(\gamma_r (T+\tau)+\tau)].$
	\end{defn}
		\begin{defn}[Generalized velocity diameters]\label{Fn}
		We define the generalized initial velocity diameter
		\begin{equation}\label{F0}
			F_0:=\max_{i,j =1,\dots, N}\max_{\sigma,s\in I_0}\lvert v_i(\sigma)-v_j(s)\rvert=\max_{i,j =1,\dots, N}\max_{\sigma,s\in [-\tau,0]}\lvert v_i(\sigma)-v_j(s)\rvert.
		\end{equation}
		For all $n\in \mathbb{N}_0$, we define the $n$-generalized velocity diameter
		\begin{equation}\label{GVD}
			F_n:=\max_{i,j =1,\dots, N}\max_{\sigma,s\in I_n}\lvert v_i(\sigma)-v_j(s)\rvert.
		\end{equation}
	\end{defn}
	
{As for the first-order model, we first  state some preliminary estimates (see \cite{CCP}).}
\begin{lem}\label{L1cs}
	For each $n\in \mathbb{N}_0$ and for any vector $v\in \RR^{d}$, we have 
	\begin{equation}\label{scalpronoffcs}
		r_n^v\leq \langle v_{i}(t),v\rangle \leq R_n^v,
	\end{equation}for all  $t\geq n(\gamma_r(T+\tau)+\tau)-\tau$ and for any $i=1,\dots,N$, where $r_n^v$ and $R_n^v$ are defined in \eqref{r_n} and \eqref{R_n}, respectively, $\gamma_r$ is the distance from the root defined in \eqref{gamma} and $T$ is the positive constant in \eqref{PE}.
\end{lem}

\begin{lem}\label{lemmadiamonoffcs}
	For each $n\in \mathbb{N}_0$, we have 
	\begin{equation}\label{diamonoffcs}
		\lvert v_i(s)-v_j(t)\rvert\leq F_n,
	\end{equation}
	for all  $s,t\geq n(\gamma_r (T+\tau)+\tau)-\tau$ and for any $i,j=1,\dots,N$, where $F_n$ is the $n$-generalized velocity diameter defined in \eqref{GVD}, $\gamma_r$ is the distance from the root defined in \eqref{gamma} and $T$ is the positive constant in \eqref{PE}.
	
	In particular, for all $n\in \mathbb{N}_0$, 	
	\begin{equation}\label{Fndec}
		F_{n+1}\leq F_n,
	\end{equation}
	and
	\begin{equation}\label{diamonoff2cs}
		d_V(t)\leq F_n,\quad \forall t\geq n(\gamma_r (T+\tau)+\tau)-\tau,
	\end{equation}
	where $d_V(\cdot)$ is the {velocity} diameter of the solution defined in \eqref{veldiam}.
\end{lem}
\begin{lem}\label{L3onoffcs}
	We have \begin{equation}\label{boundsolonoffcs}
		\lvert v_{i}(t)\rvert\leq C^V_{0}:=\max_{j=1,\dots,N}\,\,\max_{s\in [-\tau, 0]}\lvert v_{j}(s)\rvert,
	\end{equation}
	for all $t\geq -\tau$ and $i=1,\dots,N$. 
\end{lem}
	Note that, differently from the first-order model \eqref{onoff}, the uniform bound \eqref{boundsolonoffcs} does not allow us to deduce the existence of a positive bound from below on the communication rates since the arguments of the influence function are $\lvert x_i(t)-x_{j}(t-\tau_{ij})\rvert$, $i,j=1,\dots,N$. However, we can deduce the following bounds on the arguments of the influence function, that are local in time but uniform in $i,j$.
	\begin{lem}\label{lemmacs}
		For every $i,j=1,\dots,N$, we get
		\begin{equation}\label{dist}
			\lvert x_{i}(t)-x_{j}(t-\tau_{ij}(t))\rvert\leq \tau C_{0}^{V}+M^{X}_{0}+d_{X}(t), \quad\forall t\geq0,
		\end{equation}
		where $C_{0}^{V}$ is the positive constant in \eqref{boundsolonoffcs}, $d_X(\cdot)$ is the position diameter of the solution defined in \eqref{posdiam} and $M^{X}_{0}$ is the positive constant defined as
        \begin{equation}\label{M_0}
        	M_0^{X}:=\max_{i=1,\dots,N}\,\,\max_{s,t\in [-\tau,0]}\lvert x_{i}(s)-x_i(t)\rvert.
        \end{equation}       
	\end{lem}
    \begin{proof}
   	Given $i,j=1,\dots,N$ and $t\geq0$, we have
   	\begin{equation}\label{split}
   		\begin{array}{l}
   			\vspace{0.3cm}\displaystyle{\lvert x_{i}(t)-x_{j}(t-\tau_{ij}(t))\rvert\leq \lvert x_{i}(t)-x_{j}(t)\rvert+\lvert x_{j}(t)-x_{j}(t-\tau_{ij}(t))\rvert}\\
   			\displaystyle{\hspace{2cm}\leq d_{X}(t)+\lvert x_{j}(t)-x_{j}(t-\tau_{ij}(t))\rvert.}
   		\end{array}
   	\end{equation}
   	We estimate $\lvert x_{j}(t)-x_{j}(t-\tau_{ij}(t))\rvert.$ If $t-\tau_{ij}(t)>0$, from \eqref{taubounded} and \eqref{boundsolonoffcs} we have
   	$$\lvert x_{j}(t)-x_{j}(t-\tau_{ij}(t))\rvert\leq \int_{t-\tau_{ij}(t)}^{t}\lvert v_{j}(s)\rvert ds\leq C^{V}_{0}\tau_{ij}(t)\leq \tau C^{V}_{0}.$$
   	On the other hand, if $t-\tau_{ij}(t)\leq 0$, then $t\leq \tau_{ij}(t)\leq \tau$ and 
   	$$\begin{array}{l}
   		\vspace{0.3cm}\displaystyle{\lvert x_{j}(t)-x_{j}(t-\tau_{ij}(t))\rvert\leq\lvert x_j(0)-x_j(t-\tau_{ij}(t))\rvert+ \int_{0}^{t}\lvert v_{j}(s)\rvert ds}\\
   		\displaystyle{\hspace{2cm}\leq M^{X}_{0}+t C^{V}_{0}\leq M^{X}_0+\tau C^{V}_{0}.}
   	\end{array}$$
   	Therefore, in both cases,
   	$$\lvert x_{j}(t)-x_{j}(t-\tau_{ij}(t))\rvert\leq M^{X}_0+\tau C^{V}_{0},$$
   	from which \eqref{split} becomes
   	$$	\lvert x_{i}(t)-x_{j}(t-\tau_{ij}(t))\rvert\leq M^{X}_0+\tau C^{V}_{0}+d_{X}(t).$$
   	This concludes the proof.
   \end{proof}
Finally, before we prove our flocking result, Theorem \ref{uf}, we need to derive an estimate {analogous to \eqref{Bonoff} for the velocity variables.} To this aim, we introduce the following auxiliary function.
	\begin{defn}
		We define
		$$\phi(t):=\min\left\{\tilde\psi(r):r\in \left[0,\tau C^{V}_{0}+M^{X}_{0}+\max_{s\in[-\tau,t] }d_{X}(s)\right]\right\},$$
		for all $t\geq -\tau$.
	\end{defn}
	\begin{oss}
		Let us note that from \eqref{dist} we have the following non-uniform lower bound on the communication rates
		\begin{equation}\label{weightlowerbound}
			b_{ij}(t)\geq \frac{1}{N-1}{\alpha_{ij}(t)}\phi(t),\quad \forall t\geq 0,\,\forall i,j=1,\dots,N.
		\end{equation}
	\end{oss}
	
	\begin{prop}\label{lemma3'cs}
		For all $n\in \mathbb{N}_0$ and $v\in \mathbb{R}^d$, it holds
		\begin{equation} \label{B'cs}
			r^v_{n}+\tilde\Gamma_{n}(\langle v_r(n(\gamma_r(T+\tau)+\tau)),v\rangle-r^v_{n}) \leq \langle v_i(t),v\rangle\leq R^v_{n}-\tilde\Gamma_{n}(R^v_{n}-\langle v_r(n(\gamma_r(T+\tau)+\tau)),v\rangle), 
		\end{equation}
		for all $t\in I_{n+1}$ and $i =1,\dots, N$, where $I_{n+1}$ is defined in \eqref{In} and		
		\begin{equation}\label{Gammancs}
			\tilde\Gamma_{n}:=e^{-K (\frac{1}{2}(\gamma_r^2+3\gamma_r)(T+\tau)+\tau)}\left(\frac{\phi((n+1)(\gamma_r(T+\tau)+\tau))\tilde{\alpha}}{(N-1)(1+n(\gamma_r(T+\tau)+\tau)+(\gamma_r-1)T+\gamma_r\tau)^\beta}\right)^{\gamma_r},
		\end{equation}
		being $r_n^v$ and $R_n^v$ the quantities defined in \eqref{r_n} and \eqref{R_n}, respectively, $\gamma_r$ the distance from the root defined in \eqref{gamma}, $\beta$ the exponent in \eqref{PE} and $T$ and $\tilde{\alpha}$ the positive constants in \eqref{PE}.		
	\end{prop}
	\begin{oss}
		Let us note that, due to $\tilde{\alpha}\leq \frac{1}{K}$, we have $\tilde{\Gamma}_n\in (0,1)$.
	\end{oss}
	\begin{proof}
     The proof follows the same scheme of the one of Proposition \ref{lemma3onoff}, with only some little changes. We thus provide a brief sketch of the proof. Let $n\in \mathbb{N}_0$ and $v \in \mathbb{R}^d$. Let us fix an index $i=1,\dots,N$ with $i\neq r$. Let $i_0=r,i_1,\dots,i_p=i$, $1\leq p\leq \gamma_r$, be the shortest path from $r$ to $i$. Then, arguing as in Proposition \ref{lemma3onoff} we get
		\begin{equation}\label{firstgroncs}
			\langle v_r(t),v\rangle\leq R_n^v-e^{-K(\gamma_r(T+\tau)+\tau)}(R_n^v-\langle v_r(n(\gamma_r(T+\tau)+\tau)),v\rangle),
		\end{equation}
         for all $t\in [n(\gamma_r(T+\tau)+\tau),(n+1)(\gamma_r (T+\tau)+\tau)].$ 
		\\Now, for a.e. $t \in [n(\gamma_r(T+\tau)+\tau)+\tau,(n+1)(\gamma_r (T+\tau)+\tau)]$, arguing as in Proposition \ref{Bonoff} we find
		$$\frac{d}{d t}  \langle v_{i_1}(t),v\rangle \leq \frac{KN_{i_{1}}}{N-1} R_n^v-\alpha_{i_1r}(t)\frac{\phi(t)}{N-1}e^{-\tilde K(\gamma_r(T+\tau)+\tau)}(R_n^v-\langle v_r(n(\gamma_r(T+\tau)+\tau)),v\rangle)-\frac{  KN_{i_{1}}}{N-1}\langle v_{i_1}(t),v\rangle,$$
		from which Gronwall's inequality yields
		$$\langle v_{i_1}(t),v\rangle\leq 
		R^v_n-e^{- K(2\gamma_r(T+\tau)+\tau)}(R_n^v-\langle v_r(n(\gamma_r(T+\tau)+\tau)),v\rangle)\frac{1}{N-1}\int_{n(\gamma_r(T+\tau)+\tau)+\tau}^{t} \phi(s)\alpha_{i_1r}(s)ds,$$
		for all $t\in [n(\gamma_r(T+\tau)+\tau)+\tau,(n+1)(\gamma_r(T+\tau)+\tau)]$. Now, since $\phi$ is nonincreasing we can write
		$$\phi(t)\geq \phi((n+1)(\gamma_r(T+\tau)+\tau)),\quad \forall t\in [n(\gamma_r(T+\tau)+\tau)+\tau,(n+1)(\gamma_r(T+\tau)+\tau)].$$
		Hence,
		$$ 
		\begin{array}{l}
			\displaystyle{
		\langle v_{i_1}(t),v\rangle\leq R^v_n}\\
		\displaystyle{ \hspace{0.1 cm}-e^{-K(2\gamma_r(T+\tau)+\tau)} (R_n^v-\langle v_r(n(\gamma_r(T+\tau)+\tau)),v\rangle \frac{\phi((n+1)(\gamma_r(T+\tau)+\tau))}{N-1}\int_{n(\gamma_r(T+\tau)+\tau)+\tau}^{t}\alpha_{i_1r}(s)ds,}
		\end{array}
		$$
		for all $t\in [n(\gamma_r(T+\tau)+\tau)+\tau,(n+1)(\gamma_r(T+\tau)+\tau)]$. In particular, for $t\in [n(\gamma_r(T+\tau)+\tau)+T+\tau,(n+1)(\gamma_r(T+\tau)+\tau)]$, using \eqref{PE} we obtain
		\begin{equation}\label{i_1Tcs}
			\langle v_{i_1}(t),v\rangle\leq R^v_n-e^{- K(2\gamma_r(T+\tau)+\tau)}(R_n^v-\langle v_r(n(\gamma_r(T+\tau)+\tau)),v\rangle)\frac{\phi((n+1)(\gamma_r(T+\tau)+\tau))\tilde{\alpha}}{(N-1)(1+n(\gamma_r(T+\tau)+\tau)+\tau)^\beta}.
		\end{equation}
	The rest of the proof proceeds as the one of Proposition \ref{lemma3onoff}.	
	\end{proof}

	\begin{oss}
		Let us note that, analogously to estimate \eqref{D_nonoff}, from \eqref{B'cs} it comes that, for all $v\in\mathbb{R}^d$,
		\begin{equation}\label{claim1}
			R_{n+1}^v-r_{n+1}^v\leq (1-\tilde\Gamma_{n})(R_{n}^v-r_{n}^v),\quad\forall n\in \mathbb{N}_0.
		\end{equation}
		Moreover, setting $C^*:=e^{-K (\frac{1}{2}(\gamma_r^2+3\gamma_r)(T+\tau)+\tau)}\left(\frac{\tilde{\alpha}}{N-1}\right)^{\gamma_r},$ since
		\begin{equation}\label{gammaneq}
			\tilde\Gamma_{n}=C^*\left(\frac{\phi((n+1)(\gamma_r(T+\tau)+\tau))}{(1+n(\gamma_r(T+\tau)+\tau)+(\gamma_r-1)T+\gamma_r\tau)^\beta}\right)^{\gamma_r},\quad \forall n\in\mathbb{N}_0,
		\end{equation}
	we can rewrite \eqref{claim1} as follows
	\begin{equation}\label{claim1eq}
		R_{n+1}^v-r_{n+1}^v\leq \left(1-C^*\left(\frac{\phi((n+1)(\gamma_r(T+\tau)+\tau))}{(1+n(\gamma_r(T+\tau)+\tau)+(\gamma_r-1)T+\gamma_r\tau)^\beta}\right)^{\gamma_r}\right)(R_{n}^v-r_{n}^v),\quad\forall n\in \mathbb{N}_0.
	\end{equation}
In particular, since $\tilde{\Gamma}_n$ does not depend on the choice of the vector $v$, by definition of $F_n$ (see \eqref{GVD}) estimates \eqref{claim1} and \eqref{claim1eq} imply
\begin{equation}\label{F_n}
	F_{n+1}\leq (1-\tilde\Gamma_{n})F_n,\quad \forall n\in \mathbb{N}_0,
\end{equation}
and
\begin{equation}\label{F_neq}
	F_{n+1}\leq \left(1-C^*\left(\frac{\phi((n+1)(\gamma_r(T+\tau)+\tau))}{(1+n(\gamma_r(T+\tau)+\tau)+(\gamma_r-1)T+\gamma_r\tau)^\beta}\right)^{\gamma_r}\right)F_n,\quad \forall n\in \mathbb{N}_0.
\end{equation}
	\end{oss} 
    We are now able to prove our flocking result.
    \begin{proof}[Proof of Theorem \ref{uf}]
    Let $\tilde\Gamma_n$ be defined in \eqref{Gammancs}. We claim that there exist positive constants $a,b$ such that
    \begin{equation}\label{Gammanhat}
        \tilde{\Gamma}_n\geq \hat{\Gamma}_n:=\frac{C^*}{(a+bn)^{(\eta+\beta)\gamma_r}},\quad \forall n\in\mathbb{N}_0,
    \end{equation}
    where $\eta\in (0,1)$ is the constant in \eqref{infint}.
    Indeed, let $n\in\mathbb{N}_0$. Then, since \eqref{boundsolonoffcs} implies $d_X(t)\leq d_X(0)+2C_0^Vt$, for all $t\geq 0$, we have
    $$\max_{t\in[-\tau,(n+1)(\gamma_r(T+\tau)+\tau)]}d_X(t)\leq D_0^X+2C_0^V(n+1)(\gamma_r(T+\tau)+\tau),$$
    where $D_0^X$ is defined as $$D_0^X:=\max_{t\in [-\tau,0]}d_X(t).$$
    By definition of $\phi$, this gives
    $$\phi((n+1)(\gamma_r(T+\tau)+\tau))\geq \min\{\tilde\psi(r):r\in \left[0,\tau C_0^V+M_0^X+D_0^X+2C_0^V(n+1)(\gamma_r(T+\tau)+\tau)\right] \}.$$
    Using condition \eqref{infint}, we obtain
    $$\begin{array}{l}
    \vspace{0.3cm}\displaystyle{\phi((n+1)(\gamma_r(T+\tau)+\tau))\geq \min\left\{\frac{1}{(1+r)^\eta}:r\in \left[0,\tau C_0^V+M_0^X+D_0^X+2C_0^V(n+1)(\gamma_r(T+\tau)+\tau)\right] \right\}}\\
    \displaystyle{\hspace{2cm}=\frac{1}{(1+\tau C_0^V+M_0^X+D_0^X+2C_0^V(n+1)(\gamma_r(T+\tau)+\tau))^\eta}.}
    \end{array}$$
    Consequently, by definition of $\tilde\Gamma_n$,
    $$\begin{array}{l}
    \vspace{0.3cm}\displaystyle{\tilde\Gamma_n\geq C^*\frac{1}{(1+n(\gamma_r(T+\tau)+\tau)+(\gamma_r-1)T+\gamma_r\tau)^{\beta\gamma_r}}\times}\\ \vspace{0.3cm}
    \displaystyle{\hspace{3cm}\times\frac{1}{(1+\tau C_0^V+M_0^X+D_0^X+2C_0^V(n+1)(\gamma_r(T+\tau)+\tau))^{\eta\gamma_r}}}\\
    \displaystyle{\hspace{3cm}\geq \frac{C^*}{(a+b n)^{(\eta+\beta)\gamma_r}},}
    \end{array}$$
    where $$a:=1+\max\{(\gamma_r-1)T+\gamma_r\tau, \tau C_0^V+M_0^X+D_0^X+2C_0^V(\gamma_r(T+\tau)+\tau)\},$$
    $$b:=(\gamma_r(T+\tau)+\tau)\max\{1,2C_0^V\}.$$
    Hence, \eqref{Gammanhat} holds true. 
    \\Now, combining \eqref{Gammanhat} with \eqref{F_neq} we can write
    \begin{equation}\label{F_neqnuova}
    F_{n+1}\leq (1-\hat{\Gamma}_n)F_0,\quad \forall n\in \mathbb{N}_0. 
    \end{equation}
    Hence, since the terms $\hat{\Gamma}_n$ are analogous to the quantities $\Gamma_n$ defined in \eqref{Gammaonoff}, we can argue as in the the proof of Theorem \ref{consres} to get, using \eqref{diamonoff2cs} and \eqref{F_neqnuova},
    \begin{equation}\label{decadimento}
    d_V(t)\leq e^{{\sum\limits_{k=0}^{n-1}\ln (1-\hat{\Gamma}_k)}}F_0,
    \end{equation}
    for all $t\geq n(\gamma_r(T+\tau)+\tau)-\tau$, for all $n\in \mathbb{N}$. Once we have such an estimate, we can repeat the same argument in Theorem \ref{consres} to obtain, thanks to the condition $(\eta+\beta)\gamma_r<1$, the decay of the velocity diameter, namely $d_V(t)\to 0$, as $t\to +\infty$. Hence, condition (ii) in the Definition \ref{unflock} is satisfied.

    At this point, in order to obtain the asymptotic flocking for system \eqref{csp}, it remains to show that the position diameter $d_X$ is uniformly bounded. We elaborate more on the decay estimate \eqref{decadimento}. Let $n\in\mathbb{N}$ and $t\geq n(\gamma_r(T+\tau)+\tau)-\tau$. Using that $\ln(1-x)\leq -x$, for $x\in (0,1)$, and the definition of $\hat{\Gamma}_k$, we have
    $$d_V(t)\leq e^{-\sum\limits_{k=0}^{n-1}\hat{\Gamma}_k}F_0=  e^{-\sum\limits_{k=0}^{n-1}\frac{C^*}{(a+bk)^{(\eta+\beta)\gamma_r}}}F_0.$$
    Notice that
    $$\begin{array}{l}
    \vspace{0.3cm}\displaystyle{\sum_{k=0}^{n-1}\frac{C^*}{(a+bk)^{(\eta+\beta)\gamma_r}}=\sum_{k=0}^{n-1}\int_{k}^{k+1}\frac{C^*}{(a+bk)^{(\eta+\beta)\gamma_r}}dx\geq \sum_{k=0}^{n-1}\int_{k}^{k+1}\frac{C^*}{(a+bx)^{(\eta+\beta)\gamma_r}}dx}\\
    \displaystyle{\hspace{2cm}=\int_0^n \frac{C^*}{(a+bx)^{(\eta+\beta)\gamma_r}} dx=\frac{C^*}{b(1-(\eta+\beta)\gamma_r)}((a+bn)^{1-(\eta+\beta)\gamma_r}-a^{1-(\eta+\beta)\gamma_r}).}
    \end{array}$$
    Therefore, 
    \begin{equation}\label{decadimento2}
        d_V(t)\leq e^{-\frac{C^*}{b(1-(\eta+\beta)\gamma_r)}(a+bn)^{1-(\eta+\beta)\gamma_r}}e^{\frac{C^*}{b(1-(\eta+\beta)\gamma_r)}a^{1-(\eta+\beta)\gamma_r}}F_0,\quad \forall t\geq n(\gamma_r(T+\tau)+\tau)-\tau,\,\forall n\in \mathbb{N}.
    \end{equation}
    This implies that $\int_0^{+\infty}d_V(t)dt<+\infty.$ Indeed, from \eqref{diamonoff2cs} and \eqref{decadimento2}
    \begin{equation}\label{stimaint}
        \begin{array}{l}
    \vspace{0.3cm}\displaystyle{\int_0^{+\infty}d_V(t)dt=\int_0^{\gamma_r(T+\tau)}d_V(t)dt+\int_{\gamma_r(T+\tau)}^{+\infty}d_V(t)dt=\int_0^{\gamma_r(T+\tau)}d_V(t)dt+\sum_{n=1}^{\infty}\int_{n(\gamma_r(T+\tau)+\tau)-\tau}^{(n+1)(\gamma_r(T+\tau)+\tau)-\tau}d_V(t)dt}\\
    \displaystyle{\hspace{1cm}\leq F_0(\gamma_r(T+\tau)+\tau)+F_0(\gamma_r(T+\tau)+\tau)e^{\frac{C^*}{b(1-(\eta+\beta)\gamma_r)
    }a^{1-(\eta+\beta)\gamma_r}}\sum_{n=1}^{\infty}e^{-\frac{C^*}{b(1-(\eta+\beta)\gamma_r)}(a+bn)^{1-(\eta+\beta)\gamma_r}}.}
    \end{array}
    \end{equation}
    Now, $\sum\limits_{n=1}^{\infty}e^{-\frac{C^*}{b(1-(\eta+\beta)\gamma_r)}(a+bn)^{1-(\eta+\beta)\gamma_r}}<+\infty$. Indeed, 
   { \begin{equation}\label{asintotico}
        e^{-\frac{C^*}{b(1-(\eta+\beta)\gamma_r)}(a+bn)^{1-(\eta+\beta)\gamma_r}}\underset{n\to\infty}{\sim} e^{-cn^{1-(\eta+\beta)\gamma_r}}.
    \end{equation}
    for a positive constant $c.$
    Since for $t$ sufficiently large $e^{ct}\geq t^p$, for all $p>0$, we can choose $p=\frac{2}{1-(\eta+\beta)\gamma_r}>0$ to get
    $$e^{-cn^{1-(\eta+\beta)\gamma_r}}\leq \frac{1}{n^{(1-(\eta+\beta)\gamma_r)\frac{2}{1-(\eta+\beta)\gamma_r}}}=\frac{1}{n^2},$$
    for $n$ large.} Thus, since $\sum\limits_{n=1}^{\infty}\frac{1}{n^2}<+\infty$, we can conclude that $\sum\limits_{n=1}^{\infty}e^{-n^{1-(\eta+\beta)\gamma_r}}<+\infty$. Hence, from \eqref{asintotico}, $\sum\limits_{n=1}^{\infty}e^{-\frac{C^*}{b(1-(\eta+\beta)\gamma_r)}(a+bn)^{1-(\eta+\beta)\gamma_r}}<+\infty$, which gives, due to \eqref{stimaint}, $\int_0^{+\infty}d_V(t)dt<+\infty$. \\Finally, using that $\int_0^{+\infty}d_V(t)dt<+\infty$, since  
    $$d_X(t)\leq d_X(0)+\int_0^{t}d_V(t),\quad \forall t\geq 0,$$
    we deduce that $d_X$ is uniformly bounded, i.e. condition (i) in Definition \ref{unflock} is fulfilled. This concludes the proof.

    \end{proof}
	\begin{oss}[Case $\beta=0$]\label{beta=0}
			Assume $\beta=0$, namely the weights $\{\alpha_{ij}\}_{i,j=1,\dots,N}$ satisfy the {classical Persistence Excitation condition \eqref{classicalPE}.} Then the flocking result Theorem \ref{uf} holds under the weaker condition \eqref{infint2}, instead of \eqref{infint}. Moreover, the unconditional flocking is achieved exponentially fast. Indeed, in this case we can perform the Lyapunov functional approach in \cite{CCP}. We provide here just some details. 
            \\Let us define
		$$\Theta_n=\frac{\tilde\Gamma_{n}}{\gamma_r(T+\tau)+\tau}, \quad\forall n\in \mathbb{N}_0.$$
		We introduce the function $\mathcal{E}:[-\tau,+\infty)\rightarrow [0,+\infty),$ 
		$$\mathcal{E}(t):=\begin{cases}
			F_{0}, \hspace{4.5cm}t\in [-\tau,\gamma_r(T+\tau)+\tau],\\
            \mathcal{E}(n(\gamma_r(T+\tau)+\tau))\left(1-\Theta_n(t-n(\gamma_r(T+\tau)+\tau))\right), \\
            \vspace{0.2cm}\hspace{5cm}t\in (n(\gamma_r(T+\tau)+\tau),(n+1)(\gamma_r(T+\tau)+\tau)],\,n\geq 1.
		\end{cases}$$
		By definition, $\mathcal{E}$ is continuous, positive and nonincreasing. Furthermore, arguing as in \cite{CCP}, one can use an induction argument to prove that
        \begin{equation}\label{boundIn}
			F_{n}\leq \mathcal{E}(t),\quad \forall t\in [-\tau,n(\gamma_r(T+\tau)+\tau)],\,\forall n\in\mathbb{N}_0.
		\end{equation} 
	Next, we define the Lyapunov functional $\mathcal{W}:[-\tau,+\infty)\rightarrow [0,+\infty)$, 
 $$\mathcal{W}(t):=(\gamma_r(T+\tau)+\tau)\mathcal{E}(t)+C^*\int_{0}^{\tau C^{V}_{0}+M^{X}_{0}+\underset{s\in [-\tau,t+\gamma_r(T+\tau)+\tau]}{\max}d_{X}(s)}\left(\min_{\sigma\in [0,r]}\tilde\psi(\sigma)\right)^{\gamma_r}\,dr.$$
By definition, $\mathcal{W}$ is continuous. Moreover, for every $n\geq 1$ and for a.e. $t\in(n(\gamma_r(T+\tau)+\tau),(n+1)(\gamma_r(T+\tau)+\tau)) $, \eqref{diamonoff2cs} and \eqref{boundIn} imply that (see \cite{CCP} for further details)
      $$ \begin{array}{l}		\vspace{0.3cm}\displaystyle{ \frac{d}{dt}\mathcal{W}(t)=(\gamma_r(T+\tau)+\tau)\frac{d}{dt}\mathcal{E}(t)+C^*(\phi(t+\gamma_r(T+\tau)+\tau))^{\gamma_r}\,\frac{d}{dt}\underset{s\in [-\tau,t+(\gamma_r(T+\tau)+\tau)]}{\max}d_{X}(s)}\\          \displaystyle{\hspace{1cm}\leq-\mathcal{E}(n(\gamma_r(T+\tau)+\tau))\Theta_n(\gamma_r(T+\tau)+\tau)+C^*(\phi(t+\gamma_r(T+\tau)+\tau))^{\gamma_r}d_V(t+\gamma_r(T+\tau)+\tau)\leq 0.}       
      \end{array}$$
		Thus, \begin{equation}\label{negder}
        \frac{d}{dt}\mathcal{W}(t)\leq0,\quad \text{a.e. }t>\gamma_r(T+\tau)+\tau,
		\end{equation}
	which implies \begin{equation}\label{2tau1}
		\mathcal{W}(t)\leq \mathcal{W}(\gamma_r(T+\tau)+\tau),\quad \forall t\geq \gamma_r(T+\tau)+\tau.
		\end{equation}          
            Once we obtain that $\mathcal{W}$ is bounded, we can use condition \eqref{infint2} to prove the existence of a uniform bound $d^*$ on the position diameter $d_X$ using the same contradiction argument in \cite{CCP}.
            \\Finally, using that $d_X$ is uniformly bounded in \eqref{dist}, we deduce the following bound from below on the influence function
            $$\tilde\psi(\lvert x_i(t)-x_j(t-\tau_{ij}(t))\rvert)\geq \hat{\phi}:=\min_{r\in [0,\tau C_0^V+M_0^X+d^*]}\tilde\psi(r),\quad \forall t \geq 0.$$
            This last fact together with \eqref{F_neq} yields
            $$F_{n+1}\leq (1-C^*\hat{\phi}^{\gamma_r})F_n,\quad \forall n\in \mathbb{N}_0.$$
            Then, repeating the same argument in Remark \ref{casobeta0}, we can show that
            $$d_V(t)\leq  e^{-\frac{1}{\gamma_r(T+\tau)+\tau}\ln\left(\frac{1}{1-C^*\hat{\phi}^{\gamma_r}}\right)(t-\gamma_r(T+\tau)-\tau)}F_0,\quad \forall t\geq 0,$$
            where $F_0$ is the initial generalized velocity diameter. So, system \eqref{csp} exhibits unconditional flocking and the velocity diameter decays exponentially fast.  

\vspace{0.1cm}
            
            Let us also point out that, in particular, if the influence function $\psi$ is nonincreasing and $\gamma_r=1$, namely the root can directly influence all the other agents, the unconditional exponential flocking is achieved under the classical condition \eqref{unflockcond}. In general, when the influence function is not necessarily monotonic and the root could not directly influence all the other agents, the classical condition \eqref{unflockcond} is not sufficient for proving the unconditional flocking and the stronger condition \eqref{infint2} has to be required (see also \cite{CCP, Cont}).
	\end{oss}
    \noindent {\bf Acknowledgements.} {\small The authors are members of Gruppo Nazionale per l’Analisi Matematica,
    	la Probabilità e le loro Applicazioni (GNAMPA) of the Istituto Nazionale di
    	Alta Matematica (INdAM). }

\end{document}